\documentclass{amsart}
\usepackage{amssymb, epsfig,amsthm}
\usepackage[center]{caption}
\theoremstyle{plain} 
\newtheorem{theorem}{\indent\sc Theorem}[section]
\newtheorem{lemma}[theorem]{\indent\sc Lemma}
\newtheorem{corollary}[theorem]{\indent\sc Corollary}
\newtheorem{proposition}[theorem]{\indent\sc Proposition}

\theoremstyle{definition} 

\newtheorem{remark}[theorem]{\indent\sc Remark}

\def\bea{\begin{eqnarray*}}
\def\eea{\end{eqnarray*}}
\def\be{\begin{equation}}
\def\ee{\end{equation}}
\def\tr{\Delta}

\def\a{\alpha}

\begin{document}

\title{Total Scalar Curvature and Harmonic Curvature}
\author[G. Yun]{Gabjin Yun}
\author[J. Chang]{Jeongwook Chang}
\author[S. Hwang]{Seungsu Hwang} 
\thanks{The first author was supported by the Basic Science Research Program through the National Research Foundation
of Korea(NRF) funded by the Ministry of Education, Science and Technology(2011-0007465), the second author by  the Ministry of Education, Science and Technology(2011-0005235), and the third author by  the Ministry of Education, Science and Technology(2011-0005211).}
\address{Gabjin Yun\\ Department of Mathematics, Myong Ji University\\
San 38-2 Namdong, Yongin, Gyeonggi 449-728, Korea}
\email{gabjin@mju.ac.kr}
\address{Jeongwookk Chang\\  Department of Mathematics Education, Dankook University\\
126 Jukjeon-dong, Suji-gu Yongin, Gyeonggi 448-701, Korea}
\address{Seungsu Hwang (Corresponding author)\\ Department of Mathematics, Chung-Ang University\\
84 HeukSeok-ro DongJak-gu, Seoul, Korea}
\email{seungsu@cau.ac.kr}

\maketitle
\begin{abstract}
On a compact $n$-dimensional manifold, it has been conjectured that a critical point metric of the total scalar curvature, restricted to the space of metrics with constant scalar curvature of unit volume, will be Einstein. This conjecture was proposed in 1984 by Besse, but has yet to be proved. In this paper,
we prove that if the manifold with the critical point metric has harmonic curvature, then it is isometric to a standard sphere.
\end{abstract}

\section{Introduction}
Let $M$ be an $n$-dimensional compact manifold and
${\mathcal M}_1$ be the set of all smooth Riemannian structures of unit volume on $M$.
The scalar curvature $s_g$ is a non-linear function of the metric $g$.
Its linearization at $g$ in the direction of the symmetric two-tensor $h$ is given by
$$s'_g(h)=-\tr_g tr h + \delta_g\delta_g h -g(h,r_g).$$
Here, $\tr_g$ is the negative Laplacian of $g$, $r_g$ is its Ricci tensor, and $
\delta_g$ is the metric dual of the map on the bundle of symmetric tensors induced by covariant differentiation.
The dual linearized operator $s_g'^*$ of $s_g'$ is given by
\begin{equation} \label{cpe2} s'^*_g(f)=D_g df-g \Delta_g f  - fr_g,
\end{equation}
where $D_gdf$ is the Hessian of $f$.

The following formulae are important for the analysis of the map
$${\mathcal S}: g \mapsto \int_M s_g dv_g $$
when ${\mathcal S}$ is  defined over ${\mathcal C}$, a subset of ${\mathcal M}_1$, consisting of constant scalar curvature metrics.
The Euler-Lagrange equation
of ${\mathcal S}$ restricted to ${\mathcal C}$ may be written as the
following critical point equation(CPE, hereafter):
\begin{equation}\label{cp0}
z_g=s'^*_g(f),
\end{equation}
where $z_g$ is the traceless Ricci tensor defined by $z_g=r_g-\frac {s_g}n$ and $f$ is a function on $M$ with a vanishing mean value.  It is obvious that a solution $g$ of the CPE is Einstein if $f$ is trivial. Therefore, from now on, we consider only the case when $(g,f)$ is a non-trivial solution of the CPE.

In 1987 Besse proposed the following conjecture in \cite{Besse1987}:
\vskip .5pc\noindent
{\bf Conjecture I.} {\it A solution of the critical point equation (\ref{cp0}) is Einstein.}
\vskip .5pc
There are some partial answers to Conjecture I.
Among them, Lafontaine showed in \cite{Lafontaine1983} that if a solution metric $g$
of the CPE is conformally flat and $\ker s_g'^*\neq 0$, then such a metric is Einstein.
Later  Bessi{\`e}res, Lafontaine, and Rozoy showed that if a solution metric $g$
of the $3$-dimensional CPE is conformally flat, Conjecture I is true without the need for a condition on $\ker s_g'^*$ (\cite{BLR2003}).
Recently we proved the following result:
\begin{theorem}  \cite{CYH2011}  \label{thm1} Let $(g,f)$ be a non-trivial solution of the CPE on
an $n$-dimensional compact Riemannian manifold $M$.
If the Ricci tensor of $g$ is parallel, then $(M,g)$ is isometric to a standard sphere.
\end{theorem}

The geometric structure of an Einstein solution is
known to be simple;
 Obata showed that such a solution is isometric to a standard
$n$-sphere(\cite{Obata1962}).
For more details, we refer the reader to \cite{Besse1987} and \cite{HW2000}.

In this paper, we answer Conjecture I for harmonic curvature, which is a generalization of conformal flatness or parallel Ricci tensor conditions.
We say that $(M,g)$ has harmonic curvature if the divergence of the Riemmann curvature vanishes, i.e., $\delta R=0$.
It is well known that every manifold with a parallel Ricci tensor has harmonic curvature. However, there are examples of compact and noncompact Riemannian manifolds with $\delta R=0$ and $\nabla r_g\neq 0$ (see \cite{Derdzinski1982} and Theorem 5.2 in \cite{Gray1978}). By virtue of Theorem~\ref{thm1}, it is natural to ask whether a solution metric of the CPE which has harmonic curvature is Einstein. In the following we show that the answer to this question is affirmative.

\begin{theorem}\label{th3}
Let $(g,f)$ be a non-trivial solution of the CPE on
an $n$-dimensional compact Riemannian manifold $M$.
If $(M,g)$ has harmonic curvature, then $M$ is isometric to a standard sphere.
\end{theorem}
This result is a good progress to solve Conjecture I. The remaining part is to show that the solution metric in Conjecture I is harmonic. 
As an immediate consequence of Theorem~\ref{th3}, we have the following generalization of Lafontain's result(\cite{Lafontaine1983}),
which does not need a condition on $\ker s_g'^*$.
\begin{corollary}
Let $(g,f)$ be a non-trivial solution of the CPE on
an $n$-dimensional compact Riemannian manifold $M$.
If $g$ is conformally flat, then $(M,g)$ is isometric to a standard sphere.
\end{corollary}

This paper is organized as follows. In Section 2, we analyze the critical point equations. In Section 3, we formulate the critical point equations under  harmonic curvature conditions. In particular, we decompose the traceless Ricci tensor $z_g$ in Lemma~\ref{lem5} into tangential component of the level sets of $f$ and their orthogonal complements. Using this decomposition, we prove Theorem~\ref{wneinstein}, a special case of Theorem~\ref{th3}, which has the condition that one component ${\mathcal W}_N$ of the decomposition of $z_g$ vanishes. Finally, we will complete the proof of Theorem~\ref{th3} in Section 5 by showing that ${\mathcal W}_N$ is identically zero on $M$.

\section{Preliminaries}
Let $(M,g,f)$ be a non-trivial solution of the CPE.
Taking the trace of (\ref{cp0}) gives
\begin{equation}\label{eq18}
\tr_g f =-\frac {s_g}{n-1}f.
\end{equation}
Thus, by (\ref{cpe2}),  (\ref{cp0}) and  (\ref{eq18}), the CPE may be written as
\begin{equation} \label{cpe1}
(1+f)z_g=D_gdf+\frac {s_gf}{n(n-1)}g.
\end{equation}
Here the scalar curvature $s_g$ is constant from our stated assumptions.
\begin{proposition}\label{prop1-2}
Let $(g,f)$ be a non-trivial solution of (\ref{cpe1}), and $B=\{x\in M^n\,|\, f(x)=-1\}$. Then $B$ has measure zero.
\end{proposition}
\begin{proof}
Let $B'$ be the set of critical points of $f$ in $B$. Then $B\setminus B'$ is a union of
hypersurfaces. For $p \in B'$,  from (\ref{cpe1}) we have $$Ddf_p(\xi,\xi)=\frac {s_g}{n(n-1)}g_p(\xi,\xi)>0$$ for any nonzero tangent vector $\xi$ in the tangent
space $T_pM$ at $p$. Thus $p$ is a non-degenerate critical point of $f$. Such non-degenerate critical points
are isolated, and thus the set $B'$ should be finite. Therefore $B=B'\cup (B\setminus B')$ has measure zero.
\end{proof}

Let ${\rm Crit}(f)=\{x\in M \,\vert\, df(x)=0\}$. For this set, we observe the following fact.
\begin{proposition} \label{prop01} The measure of  ${\rm Crit}(f)$ is zero.
\end{proposition}
\begin{proof}
We first claim that ${\rm Crit}(f)\cap B$ is finite. If $p\in \, {\rm Crit}(f) \cap B$,
we observe that as in the proof of Proposition~\ref{prop1-2} that $p$ is an isolated point, and thus ${\rm Crit}(f)\, \cap B$ is finite since $B$ is compact.
To prove that there is no open set in ${\rm Crit}(f)\,\cap (M\setminus B)$, it suffices to prove that $g$ and $f$ are analytic in $M\setminus B$.
Since the coefficients, in local harmonic coordinates, of the coupled elliptic system (\ref{eq18}) and (\ref{cpe1})  are
real-analytic, it follows from Theorem 6.6.1 in \cite{Morrey1966} that $g$ and $h=1+f$ are in fact real-analytic where $f\neq -1$.
\end{proof}

\section{Harmonic Curvature}
In this section we study harmonic curvature and its relation to the CPE.
For an $n$-dimensional Riemannian manifold $(M,g)$, the second Bianchi identity yields the well-known divergence formulae $\delta r_g =-\frac 12 d{s_g}$ and
\begin{equation}\label{eq32}\delta R=-d^D r_g,\end{equation}
where $d^D$ is the first-order differential operator from $C^{\infty}(S^2M)$ into $C^{\infty}(\Lambda^2 T^*M\otimes T^*M)$ defined by
$$ d^D \omega (x,y,z)=D_x \omega (y,z) -D_y \omega (x,z)$$
for a two form $\omega$.
Consequently
$$\delta {\mathcal W}= -\frac {n-3}{n-2} \, d^D\left(r_g-\frac {s_g}{2(n-1)}g\right), $$
where ${\mathcal W}$ is its Weyl conformal curvature tensor.

We say that $(M,g)$ has harmonic curvature if the divergence of the Riemmann curvature vanishes, i.e., $\delta R=0$.
When $n=3$, $(M,g)$ is conformally flat and has constant scalar curvature.
When $n\geq 4$, $(M,g)$ has harmonic Weyl tensor ($\delta {\mathcal W} =0$) and constant scalar curvature.
Moreover, it is equivalent to $d^D r_g=0$, in other words, $r_g$ is a Codazzi tensor.

The product of a $1$-form $\beta$ and a symmetric two form $\eta$ can be defined by $\beta \wedge \eta (x,y,z)=\beta(x)\eta (y,z)-\beta(y)\eta(x,z)$.
Then we have the following equation.
\begin{lemma} \label{lem2} Let $(M,g)$ have harmonic curvature and $(g,f)$ be a non-trivial solution of the CPE. Then
\begin{equation} \label{eqnn02}
(n-2)\,\tilde{i}_{\nabla f }{\mathcal W} = (n-1) df \wedge z + i_{\nabla f }z\wedge g. \end{equation}
Here, we define $\tilde{i}$  by $\tilde{i}_{\xi} \omega (X,Y,Z)=\omega(X,Y,Z,\xi)$ for a $4$-tensor $\omega$, and $i$ by $i_{\xi}z(X)=z(\xi, X)$.
\end{lemma}
\begin{proof}
From (\ref{cpe1}) we have
\bea (1+f)d^Dr_g (X,Y,Z)&=& \langle R(\nabla f,Z)Y,X\rangle +\left(\frac {s_g}{n-1}df\wedge g -df \wedge r_g\right)(X,Y,Z).
\eea
Now from
\bea {\mathcal W}(X,Y,Z,W)&=& R(X,Y,Z,W)-\frac 1{n-2}(g(X,Z)r(Y,W)+g(Y,W)r(X,Z)\\
& &-g(Y,Z)r(X,W)-g(X,W)r(Y,Z))\\
& &+\frac {s_g}{(n-1)(n-2)}(g(X,Z)g(Y,W)-g(Y,Z)g(X,W)),
\eea
\[ \tilde{i}_{\nabla f}R= \tilde{i}_{\nabla f}{\mathcal W} -\frac 1{n-2}i_{\nabla f}r_g \wedge g  +\frac {s_g}{(n-1)(n-2)}df\wedge g -\frac 1{n-2}df\wedge r_g. \]
We can obtain
\bea (1+f)d^Dr_g &=& \tilde{i}_{\nabla f}{\mathcal W} -\frac 1{n-2}i_{\nabla f}r_g \wedge g  +\frac {s_g}{(n-2)}df\wedge g -\frac {n-1}{n-2}df\wedge r_g\\
&=& \tilde{i}_{\nabla f}{\mathcal W} -\frac {n-1}{n-2}df\wedge z -\frac 1{n-2}i_{\nabla f}z\wedge g .
\eea
Then equation (\ref{eqnn02}) follows from the harmonicity of the metric $g$.
\end{proof}
Throughout the rest of this paper,
{\it we assume that $(M,g)$ has harmonic curvature and $(g,f)$ is a non-trivial solution of the CPE.}
As immediate consequences of (\ref{eqnn02}), we have the following two results.
\begin{lemma} \label{lem4} For each regular value $c$ of $f$ and a tangent vector $X$ to $f^{-1}(c)$,
\be \label{eqnn3} z(X,\nabla f)=0 \ee
on $f^{-1}(c)$.
\end{lemma}

\begin{proof}
Note that $X$ is orthogonal to $\nabla f$.
Applying the triple $(X,\nabla f, \nabla f)$ into (\ref{eqnn02}) gives
$$ 0=(n-2){\mathcal W}(X,\nabla f,\nabla f, \nabla f)= (2-n)z(X,\nabla f)|df|^2, $$
since $df(X)=0$. This implies that (\ref{eqnn3}) is true.
\end{proof}

\begin{lemma} \label{lem3} On $M$ we have
\begin{equation} |df|^2 i_{\nabla f}z=z(\nabla f, \nabla f) \, df. \label{eqt1}\end{equation}
\end{lemma}

\begin{proof} From $d^Dr_g=0$ and (\ref{eqnn02}),  applying the triple $(\nabla f, Y,\nabla f)$ into (\ref{eqnn02}) with an arbitrary vector $Y$ gives
\bea 0&=&(n-2){\mathcal W}(\nabla f, Y, \nabla f,\nabla f)\\
&=& (n-1)(|df |^2 z(Y,\nabla f) -df (Y) z({\nabla f},\nabla f))\\
& & +z(\nabla f, \nabla f)df(Y)-z(\nabla f, Y)|df|^2.
\eea
Therefore we obtain
$$
|df|^2 z(Y,\nabla f)= df(Y) z(\nabla f,\nabla f).
$$
Thus equation (\ref{eqt1}) holds.
\end{proof}
 On $M\setminus {\rm Crit}(f)$, we can define $N:=df/|df|$ and $\alpha=z(N,N)$. Then we can rewrite the equation (\ref{eqt1}) as
\be i_{\nabla f}z=\alpha \, df .\label{eqnn1-1}\ee
Taking the divergence of (\ref{eqnn1-1}) gives
\begin{equation}\label{eqt2}\delta (i_{\nabla f}z)=
-\langle d\alpha, df\rangle -\alpha \tr f. \end{equation}
On the other hand, if $\{E_i\}_{i=1,...,n}$ is a local orthonormal basis of vector fields,
\begin{eqnarray*}
\delta (r_g(d\varphi, \cdot))&=& -\sum_i (D_{E_i}r_g(d\varphi))(E_i)=-\sum_i E_i(r(d\varphi,E_i))\\
&=&-\langle Dd\varphi, r_g\rangle +\delta r_g (d\varphi)=-\langle Dd\varphi, r_g\rangle-\frac 12 \langle d{s_g}, d\varphi\rangle
\end{eqnarray*}
for any smooth function $\varphi$.
Thus, since ${s_g}$ is constant, from (\ref{cpe1}) we have
\begin{equation}\label{eqt4} \delta (i_{\nabla f}z)=
\delta(r(df,\cdot))+\frac {s_g}n \tr f = -\langle Ddf, r\rangle -\frac {{s_g}^2}{n(n-1)} f =-(1+f)|z|^2.
\end{equation}
Therefore, by (\ref{eqt2}) and (\ref{eqt4}) we obtain
\begin{equation}\label{eqt3}
(1+f)|z|^2= -\frac {s_g}{n-1}\,\alpha  f+\langle d\alpha, df\rangle.
\end{equation}

On the other hand,
by Lemma~\ref{lem2} and Lemma~\ref{lem3}  we have
$$ -df \wedge z = \frac {\alpha}{n-1}\, df\wedge g -\frac {n-2}{n-1} \, \tilde{i}_{\nabla f }{\mathcal W}. $$
Thus, applying the triple $(E_i, \nabla f, E_j)$ into the above equation gives the following orthogonal decomposition of $z$.
\begin{lemma}\label{lem5} Let $\{ {E_i\}}_{i=1}^n $ be a local orthonormal frame field on $M\setminus {\rm Crit}(f)$ with $E_n=N=\nabla f/|\nabla f|$. Then we have
\be \label{zdecomp1} z_{ij}=-\frac {\alpha}{n-1}\delta_{ij}-\frac {n-2}{n-1}{{\mathcal W}}_{ij},\ee
for $i,j=1,...,n-1$, where ${{\mathcal W}}_{ij}= {\mathcal W}_N(E_i,E_j)\equiv {\mathcal W}(E_i,N,E_j,N)$.
Thus
\be \label{zdecomp2} |z|^2 = \frac{n}{n-1} \alpha^2 + \left(\frac{n-2}{n-1}\right)^2|{\mathcal W}_N|^2.\ee
\end{lemma}
\noindent Note that equation
(\ref{zdecomp2}) follows from ${\rm tr} \, {\mathcal W}_N=0$.
 For a real number $c$, we denote $L_{c}$ as a connected component of $f^{-1}(c)$.
The following lemma implies that the functions $|z|^2$ and $|{\mathcal W}_N|^2$ are constant on $L_{c}$ due to (\ref{eqt3}) and (\ref{zdecomp2}) since $\alpha, |df|, \langle d\alpha,df\rangle$ are constant on $L_{c}$.

\begin{lemma}\label{lem2011-4-5-1}
For each regular value $c$ of $f$, $\alpha$, $|df|$, and $\langle d\alpha, df\rangle$ are constant on $L_{c}$.
\end{lemma}
\begin{proof}  From (\ref{cpe1}) and Lemma~\ref{lem4}, it is easy to see that $D_NN=0$;
$$D_NN=\sum_{i=1}^{n-1}\langle D_NN,E_i\rangle E_i= \frac 1{|df|}\sum_{i=1}^{n-1}\langle D_Ndf,E_i\rangle E_i =\frac {1+f}{|df|}\sum_{i=1}^{n-1}z(N,E_i) E_i=0.$$
Now, since $d^Dr_g = 0$ and the scalar curvature $s_g$ is constant, $d^D z = 0$. Thus for a tangent vector
$X$ to $L_c$, from (\ref{eqnn3}) and $D_N N = 0$ we obtain
\bea
X(\a) &=& X(z(N, N)) = D_X z(N, N) + 2z(D_X N, N)\\
&=&
D_Nz(X, N) = N(z(X, N)) - z(D_N X, N) - z(X, D_N N)\\
&=&
- \langle D_N X, N\rangle \, \a = 0.
\eea
This implies the constancy of $\alpha$ on $f^{-1}(c)$. In particular, $d\alpha =N(\alpha)N$. The second part follows easily; for a tangent vector $\xi$ to the level set of $f$,
$$\xi |df|^2=2\langle D_{\xi}df, df\rangle =2(1+f)z(\xi,df)-\frac {2sf}{n(n-1)}\xi(f)=0$$
by Lemma~\ref{lem4}. Also, since $\alpha$ is constant on each level sets of $f$,
the third statement follows from
\bea 0&=& NX(\alpha)=N\langle d\alpha, X\rangle = \langle D_Nd\alpha, X\rangle\\
&=&\langle D_X d\alpha, N\rangle =X\langle d\alpha, N\rangle,
\eea
where we used Lemma~\ref{lem4} in the last equality.
\end{proof}
\par
It is clear that $\alpha$ and $d\alpha$ are defined on $M\setminus \mbox{\rm Crit}(f)$.
For the rest of this section, we discuss the extension of $\alpha$ and $d\alpha$ onto all of $M$.
Since we have $|\alpha| \le |z|$, $\alpha$ can even be defined on the measure zero set ${\rm {Crit}}(f)$. In particular, if $z(x_0)=0$ for $x_0\in \mbox{\rm Crit}(f)$, $\alpha(x_0)$ can be continuously defined to be zero since $\lim_{x\to x_0} |\alpha|\leq \lim_{x\to x_0}|z| =0$.

Let $G=|df|^2+\frac s{n(n-1)}f^2$. It is easy to see that $G$ is continuous on all of $M$. For the function $G$, we have the following result.
\begin{lemma}  \label{lemg}For the function $G=|df|^2+\frac s{n(n-1)}f^2$ and the conformal metric  $\tilde {g}= h^{-2}g$ with $h=1+f$, we have
\be \label{glap}\tilde{\tr} G+\frac{(n-3)}h\, \tilde{g}(dG,df)= 2h^4|z|^2,\ee
where $\tilde{\tr}$ is the Laplacian  of $\tilde{g}$.
In particular, if $x_0\in \mbox{\rm Crit}(f)$ and $z(x_0)\neq 0$, then $G$ has its local minimum at $x_0$.
\end{lemma}
\begin{proof}
Making a conformal change (c.f. see \cite{BLR2003}) gives
$$ \tilde{\tr} \left(|df|^2+\frac s{n(n-1)}f^2\right)=2h^4|z|^2-2(n-3)h^2z(df,df), $$
since
\begin{eqnarray*} \tilde{\tr}|df|^2&=& 2h^4|z|^2-\frac {2s}{n(n-1)}h^2|df|^2+\frac {2s^2}{n(n-1)^2}f^2h^2\\
& & -2(n-3)h^2z(df,df)+\frac {2(n-2)}{n(n-1)}sfh|df|^2
\end{eqnarray*}
and
\[ \tilde{\tr}f^2 = -\frac {2s}{n-1}h^2f^2 +2|df|^2h^2-2(n-2)fh|df|^2.\]
Thus the function $G$ satisfies (\ref{glap}).

Note that, for any tangent vector $\xi$  at $x_0$,
\be\label{tangent}\xi (G)=2\langle D_{\xi}df, df\rangle +\frac {2sf}{n(n-1)}\langle \xi,df\rangle
     = 2hz(\xi, df).\ee
Therefore $dG=2h \, i_{\nabla f} z$.

 If $x_0\in \mbox{\rm Crit}(f)$, we have $\xi(G)(x_0)=0$ by (\ref{tangent}). Also by (\ref{glap}) and the assumption that $z(x_0)\neq 0$, $\tilde{\tr}G=2h^4|z|^2>0$ at $x_0$. Since $G$ is constant on each level sets of $f$, we may conclude that $G$ has its local minimum at $x_0$. 
\end{proof}
\begin{remark}\label{rmkm}
By Lemma~\ref{lemg} we can apply the maximum principle to  $G=|df|^2+\frac s{n(n-1)}f^2$ on the open set $M^{\epsilon}=\{x\in M\, \vert\, 1+f(x)>\epsilon\}$ for an arbitrary small positive number $\epsilon$ to conclude that $G$ achieves its maximum on $B=\{x\in M\, \vert\, f(x)=-1\}$. Similarly, we may conclude that $G$ on the set $M_{-\epsilon}=\{x\in M\, \vert\, 1+f(x)<-\epsilon\}$ also achieves its local maximum  on $B$.
\end{remark}
 Let $x_1\in \mbox{\rm Crit}(f)$. Note that $\mbox{\rm Crit}(f)$ has measure zero by Proposition~\ref{prop01} .
As mentioned above, if $z(x_1)=0$, then $\alpha(x_1)$ can be continuously defined to be zero. Now we assume that $z(x_1)\neq 0$. Let $M^0=\{x\in M\,\vert\, 1+f(x)>0\}$.

\begin{lemma}\label{lemfup0} Let $x_1\in \mbox{\rm Crit}(f)$. If $x_1\in M^0$ and $z(x_1)\neq 0$, then $x_1$ has to be a local maximum point of $f$.  Also, if $f(x_1)<-1$ and $z(x_1)\neq 0$, then $x_1$ has to be a local minimum point of $f$. If $f(x_1)=-1$, then $x_1$ has to be a local minimum point of $f$.
\end{lemma}
\begin{proof}
First of all, note that, for a connected level set $L_{f(x_1)}$ containing $x_1$, we have $|df|(y)=0$ for every $y\in L_{f(x_1)}$ and $G$ has its local minimum on $L_{f(x_1)}$ by Lemma~\ref{lemg}.

First we prove the case when $x_1\in M^0$. Suppose $x_1$ is not a local maximum point of $f$.
 Let $x_2$ be the  (global) maximum point of $f$. Then $G$ has its local minimum at $x_2$;
near $x_2$ we have
$$dG=2\, i_{\nabla f}z=2h\, \alpha \, df,$$
and $h\alpha$ is negative since, for a small connected neighborhood $\Omega^{\epsilon}=\{x\in M\, \vert\, f(x)>f(x_2)-\epsilon\}$ of $x_2$ with an arbitrarily small $\epsilon>0$ (see Fig 1 (a) where $x_2$ replaces $x_1$), $h=1+f>0$ on $\Omega^{\epsilon}$, and $$0<\int_{\Omega^{\epsilon}}(1+f)|z|^2=-\int_{\partial\Omega^{\epsilon} }\alpha |df| =-\alpha\int_{\partial\Omega^{\epsilon} } |df|,$$ implying that $\alpha <0$ on $\partial\Omega^{\epsilon}=L_{f(x_2)-\epsilon}$.
Here, we used the fact that $\int_{\Omega^{\epsilon}}(1+f)|z|^2\neq 0$; otherwise, since
the metric $g$ and thus the traceless Ricci tensor $z$ are analytic on $M^0$ as seen in the proof of Proposition~\ref{prop01}, $z\equiv 0$ on $\Omega^{\epsilon}$, which implies that $z\equiv 0$ on the connected component of $M^0$ containing $x_1$. This contradicts  the fact that $z(x_1)\neq 0$ for $x_1\in M^0$.

Now consider a geodesic $\gamma$ from $x_1$ to $x_2$  in $M^0$ (see Fig 1 (b) where $x_1$ replaces  $x_0$ and $x_2$ replaces $x_1$).
Since $G$ also has its local minimum at $x_2$ by the above argument and the fact that $G$ is constant on the each level sets of $f$, there exists a point on $\gamma$ at which $G$ has its local maximum. Then, by the maximum principle of $G$ mentioned in Remark~\ref{rmkm}, $G$ has to be constant along $\gamma$, and thus by (\ref{glap}) $z=0$ on the connected subset of $\{x\in M\, \vert\, f(x_1)\leq f(x)\leq f(x_2)\}$ containing $\gamma$, contradicting our assumption that $z(x_1)\neq 0$.
This  completes the proof of the first statement.

The proof of the second statement is similar.  The remaining case is when $f(x_1)=-1$. In this case, $x_1$ has to be a local minimum point of $f$ by Proposition~\ref{prop1-2}.
 \end{proof}

 By Lemma~\ref{lemfup0}, we may conclude that any critical point $x_1$ of $f$ with $z(x_1)\neq 0$ should be a local maximum if $f(x_1)>-1$, or a local minimum if $f(x_1)\leq -1$, and since $\alpha$ is constant on each level sets of $f$,
 $\alpha=z(df,df)/|df|^2$ can be continuously defined up to $x_1$ if $z(x_1)\neq 0$. If $z(x_1)=0$ for $x_1\in \mbox{\rm Crit}(f)$, $\alpha=0$ as discussed above.
In other words, $\alpha$ can be extended to a $C^0$-function on on all of $M$.
Also the differentiation of $\alpha$ on ${\rm {Crit}}(f)$ can be considered  in the distribution sense.

\section{A Special Case}
In this section we prove Theorem~\ref{th3} in the case when ${\mathcal W}_N= 0$ on $M\setminus {\rm Crit}(f)$.
Then by continuity ${\mathcal W}_N\equiv 0$ on $M$, and thus, by (\ref{zdecomp2}) $\alpha$ is smooth on all of $M$.
Due to the results obtained by Obata (\cite{Obata1962}), it is sufficient for the proof of Theorem~\ref{th3} to prove that $z=0$ identically on $M$.

\begin{theorem}\label{wneinstein}
If ${\mathcal W}_N\equiv 0$ on $M$, then $g$ is Einstein.
\end{theorem}
In the following we shall prove that $\alpha$ is constant by showing that $\alpha$ is superharmonic on $M$ (Lemma~\ref{spha}).
If $\alpha$ is constant,  since, by (\ref{eqt4}), (\ref{zdecomp2}), and the fact that $\int_M f=0$,
$$ \frac n{n-1}\, \alpha^2=\frac n{n-1}\int_M (1+f)\, \alpha^2= \int_M (1+f)|z|^2 =-\int_M \delta (i_{\nabla f}z) =0,$$
implying that $\alpha\equiv 0$ on $M$.

 When ${\mathcal W}_N\equiv 0$, by (\ref{eqt3}) and (\ref{zdecomp2}) we have
\be \label{wneq1}\frac n{n-1}(1+f)\, \alpha^2 =-\frac s{n-1}\, \alpha f +N(\alpha)|df|, \ee
since $N(\alpha)|df|=\langle d\alpha, df\rangle$. By virtue of Lemma~\ref{lem2011-4-5-1}, we denote $\alpha'=N(\alpha)$ and $\alpha''=NN(\alpha)$.

For the proof of Theorem~\ref{wneinstein}, we need the following Lemma~\ref{lemwn1} and Lemma~\ref{lemwn2}.
\begin{lemma} \label{lemwn1}
If ${\mathcal W}_N\equiv 0$ on $M$, we have the following equalities
\begin{eqnarray} \alpha' &=& \frac n{n-1}\, \alpha\, \delta N,\label{wneqn3}\\
 \alpha''&=& \frac {n\alpha}{n-1}\left( \alpha +\frac s{n}\right) +\frac {n+1}{n-1}\, \alpha ' \, \delta N,\label{wneqn4}\\
\tr \,\alpha &=& \frac {n\alpha}{n-1}\left( \alpha +\frac s{n}\right) +\frac {2}{n-1}\, \alpha ' \, \delta N. \label{wneqn5}
\end{eqnarray}
\end{lemma}
\begin{proof}
From the definition of divergence and  $D_NN=0$,
\bea \delta N&=& -\sum_{i=1}^{n-1}\langle D_{E_i}N,E_i\rangle\\ &=&-\frac 1{|df|}\left((1+f)\sum_{i=1}^{n-1}z(E_i,E_i)-\frac {sf}{n}\right)
= \frac 1{|df|}\left( (1+f)\alpha +\frac {sf}n\right).
\eea
Thus by (\ref{wneq1}) we obtain
$$ \alpha'|df|= \frac n{n-1} \, {\alpha}\, \left( (1+f)\alpha+\frac sn f \right)=\frac n{n-1}\, \alpha \, |df| \, \delta N.$$
Taking the derivative in the direction $N$ of (\ref{wneq1}) gives
$$ (n-1) \alpha''= n\alpha^2+s\, \alpha +(n+1) \alpha' \delta N, $$
where we used (\ref{wneqn3}) and the fact that
\be\label{dfn} N(|df|)= \langle D_Ndf,N\rangle=(1+f)\alpha -\frac {sf}{n(n-1)}.\ee
The last equation for the Laplacian of $\alpha$ follows from the following observation
$$ \tr \,\alpha =-\delta(d\alpha)=-\delta (\alpha' N)=\alpha''-\alpha' \, \delta N. $$
\end{proof}
\begin{remark}\label{rmkdf}  Equation (\ref{dfn}) holds on $M\setminus \mbox{\rm Crit}(f)$ without any condition on ${\mathcal W}_N$.
\end{remark}
The following is a special case of Lemma~\ref{analytic}. However, we include the proof for the sake of the completeness.
\begin{lemma}\label{lemwn2}
If ${\mathcal W}_N\equiv 0$ on $M$, $\alpha \leq 0$ on $M$.
\end{lemma}

\begin{proof}
Let $p$ be a maximum point of $\alpha$. Then $\alpha'(p)=0$ and $\tr \, \alpha (p)\leq 0$.
From
$$ 0\geq \tr \,\alpha (p)=\alpha''(p) =\frac n{n-1}\,\alpha^2(p)+\frac s{n-1}\,\alpha(p),$$
$(\alpha(p))^2 +\frac sn\, \alpha(p)\leq 0$, which implies that $-\frac sn \leq \alpha(p)\leq 0$. Thus we may conclude that
$\alpha$ is always non-positive on $M$.
\end{proof}

Let
$H= \{x \in M \,\vert \, \tr \, \a (x) \leq 0\,\}$ and $\Omega = \{x\in M \,\vert \, \a(x) + \frac{s}{n} <0\,\}$.
The following lemma gives us a good understanding of the set $\Omega$.
\begin{lemma} \label{holem} If ${\mathcal W}_N\equiv 0$ on $M$, $\overline{\Omega}\neq M$ and $M\setminus \Omega \varsubsetneq H$.
In particular, if $H\varsubsetneq M$, $\Omega$ can be written as the disjoint union of $M\setminus H$ and $\Omega \cap H$.
\end{lemma}

\begin{proof}
Note that $\overline{\Omega} \varsubsetneq M$; otherwise, on the set $M^{+1}=\{x\in M\, \vert\, 1+f(x)> 1\}$
$$\delta N = \frac{1}{|df|} \left[\a + f(\a+\frac{s}{n})\right] < 0,$$
and thus we obtain the following contradiction;
$$ 0> \int_{M^{+1}} \delta N =-\int_{M^{+1}}{\rm div}(N) =-\int_{\partial M^{+1}}\langle N, -N\rangle ={\rm vol}(\partial M^{+1}).$$
Thus $M\setminus \Omega$ is a non-empty set. In particular, $M\setminus \Omega$ is a subset of $H$. For the proof of this fact, we need to show that $\tr \,\alpha \leq 0$ on $M\setminus \Omega$. This follows  from  (\ref{wneqn5}) and the facts that
$\alpha \leq 0$ and $\alpha' \delta N =\frac n{n-1}\alpha (\delta N)^2\leq 0$ on $M$
by  Lemma~\ref{lemwn2} and  (\ref{wneqn3}).

Also note that we have $\alpha =-\frac sn$ on $\partial \Omega$ and thus
$$ \alpha'=\frac n{n-1}\, \alpha\, \delta N =\frac n{n-1}\frac {\alpha^2}{|df|} = \frac {s^2}{n(n-1)}\frac 1{|df|} >0$$
on $\partial \Omega$. Therefore, the outward unit normal to $\partial \Omega$ is given by
$$ \frac {\nabla \alpha}{|\nabla \alpha|}= \frac {\alpha' N}{|\alpha' N|}=\frac {\alpha'}{|\alpha'|}N=N.$$
The positivity of $\alpha'$ on $\partial \Omega$ also implies that
$M\setminus \Omega \neq H$, since  $\delta N=-\frac sn \frac 1{|df|}<0$ on $\partial \Omega$ and by (\ref{wneqn5})
$$\tr \alpha =\frac 2{n-1}\, \alpha'\, \delta N <0.$$
\end{proof}

\begin{lemma} \label{spha} If ${\mathcal W}_N\equiv 0$ on $M$, $H=M$.
\end{lemma}
\begin{proof} By Lemma~\ref{holem},  $M\setminus H \subset \Omega$.
We claim that $\Omega$ has measure zero, implying the proof of our lemma.\par
Suppose that the $n$-dimensional measure of $\Omega$ is positive.
First we observe that $\alpha'\neq 0$ in $\Omega\cap H$; if there is a point $x_0\in \Omega \cap H$ such that $\alpha'(x_0)=0$, then, since $\alpha <-\frac sn$ on $\Omega$
$$\tr \alpha(x_0) =\frac n{n-1} \, \alpha (x_0) \left( \alpha (x_0)+\frac sn\right) >0, $$
contradicting the fact that $\tr \,\alpha \leq 0$ on $H$.
Thus, from the fact that $\alpha'>0$ on $\partial \Omega$, $\alpha'>0$ on $\Omega \cap H$. For $x\in M\setminus H= \Omega \setminus (\Omega \cap H)$, $\tr \,\alpha (x) >0$, and, since $\alpha$ is constant on each level sets of $f$,
we have $\alpha'\geq 0$ on the all of $\Omega$; otherwise there exists some point $y\in M\setminus H$ such that $\alpha$ has a local maximum at $y$, which is impossible by the maximum principle.
Also, due to the fact that
$\alpha'>0$ on $\partial \Omega$, $\int_{\Omega} \alpha'>0$.

Note that by the proof of Lemma~\ref{holem} the outward unit normal vector to $\partial \Omega$ is $N$.
Now by (\ref{wneqn3})
\[
\int_{\Omega} \delta (\alpha N)= \int_{\Omega} -\langle d\alpha, N\rangle +\alpha \,\delta N
 =-\frac 1n \int_{\Omega} \alpha'. \]
On the other hand, by the divergence theorem
\[ \int_{\Omega} \delta (\alpha N)= \int_{\partial \Omega} -\alpha. \]
Consequently, by Lemma~\ref{lemwn2}
$$ 0<\frac 1n \int_{\Omega} \alpha' =\int_{\partial \Omega} \alpha\,\leq 0, $$
which is a contradiction. This completes the proof of Lemma~\ref{spha}.
\end{proof}

\section{The proof of Theorem~\ref{th3}}
This section is devoted to the proof of Theorem~\ref{th3}. Due to Theorem~\ref{wneinstein}, it suffices to prove that ${\mathcal W}_N$ vanishes identically on $M$. More precisely,
\begin{theorem}\label{thmzero}
Let $(g, f)$ be a nontrivial solution of the {\rm CPE}.
 Assume also that $(M, g)$ has harmonic curvature. Then ${\mathcal W}_N = 0$.
\end{theorem}
We first show that $\alpha$ is nonnegative on the whole space of $M$ in the following lemma, and then prove that ${\mathcal W}_N=0$  using Lemma~\ref{lem2011-4-5-4} and Lemma~\ref{lem2011-5-7-2}.

\begin{lemma}\label{analytic}
Let $(g, f)$ be a nontrivial solution of the {\rm CPE}.
 Assume also that $(M, g)$ has harmonic curvature. Then $\alpha\leq 0$.
\end{lemma}
\begin{proof}
Suppose that $\alpha(x_0)=\max_{x\in M} \alpha(x)>0$.
\vskip .5pc
{\bf Claim 1.} We have $-1<f(x_0)<0$. \par
\begin{proof}[of Claim 1]
At $x_0$, we have by (\ref{eqt3})
\be \label{criticalalpha} (1+f)|z|^2=-\frac {sf}{n-1}\, \alpha.\ee
Thus $f(x_0)\neq 1$. Also $f(x_0)\neq 0$, otherwise $|z|^2(x_0)=0$ implying that $\alpha(x_0)=0$, a contradiction.
Moreover,  at $x_0$ by (\ref{criticalalpha})
$$ 0< \frac s{n-1}\, \alpha = -\frac {1+f}{f}\, |z|^2, $$
which implies our claim is true.
\end{proof}

\begin{figure}[htb]
\centering{\scalebox{0.6}{\includegraphics{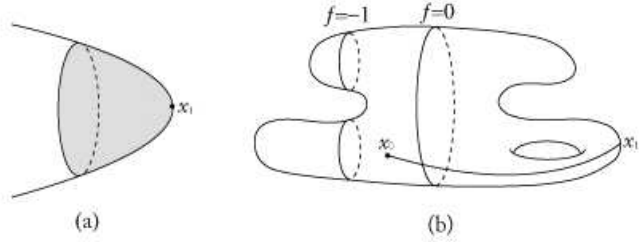}} \caption{Near a critical point}
\label{figure1-1}}
\end{figure}

Now we are ready to derive our contradiction.
Choose a maximum point $x_1$of $f$ such that a geodesic $\gamma$ from $x_0$ to $x_1$ lies entirely in $M^0$ (see Fig 1. (b)). It is easy to see that  $\alpha <0$ near every maximum points of $f$ as in the proof of Lemma~\ref{lemfup0}.
 Thus there exist points in $\gamma$ such that $\alpha=0$ at those points.  Among those points let $x_2$ be the point at which $\alpha$ changes sign.
In other words, $\alpha(x_2)=0$ and $\alpha>0$ before the point $x_2$ and $\alpha<0$ after $x_2$ along  the geodesic $\gamma$ from $x_0$ to $x_1$. To obtain a contradiction, we need to look into two cases, $|df|(x_2)\neq 0$ and $|df|(x_2)= 0$.
\vskip .5pc
{\it Case I.} $|df|(x_2)\neq 0$.
\vskip .5pc
Consider $L_{f(x_2)}$ with $x_2\in L_{f(x_2)}$. Then, since $|df|(x_2)\neq 0$,
there is a connected subset $\Omega_{\epsilon'}$ of $M$ for a sufficiently small $\epsilon'>0$ such that $f(x_2)-\epsilon'<f(x)<f(x_2)$ for every $x\in \Omega_{\epsilon'}$ with smooth boundaries $\partial \Omega_{\epsilon'}=L_{f(x_2)-\epsilon'}\cup L_{f(x_2)}$. Since
\bea 0<\int_{\Omega_{\epsilon'}}(1+f)|z|^2 &=&\int_{L_{f(x_2)}}\alpha |df| -\int_{L_{f(x_2)-\epsilon'}}\alpha |df| \\
&=& -\int_{L_{f(x_2)-\epsilon'}}\alpha |df|  <0,\eea
which is the desired contradiction.
\vskip .5pc
{\it Case II.} $|df|(x_2)=0$.
\vskip .5pc
Note that the critical point $x_2$ cannot be a local maximum point of $f$; otherwise
on the small connected neighborhood $\Omega^{\epsilon}$ of $x_2$ given by $\Omega^{\epsilon}=\{x\in M\, \vert\, f(x)> f(x_2)-\epsilon\}$ with a sufficiently small $\epsilon$ (see Fig. 1 (a) where $x_2$ replaces $x_1$),
$\alpha$ has to be positive and negative at the same time on $\partial \Omega^{\epsilon}$, which is impossible since $\alpha$ is constant on  $\partial \Omega^{\epsilon}=L_{f(x_2)-\epsilon}$.
 Similarly, $x_2$ cannot be a local minimum point of $f$. Thus by Lemma~\ref{lemfup0}, $z(x_2)=0$ and $f$ is increasing nearby $x_2$ along $\gamma$ from $x_0$ to $x_1$.

If one can find a connected subset $\Omega_{\epsilon}$ of $M$ such that $\partial \Omega_{\epsilon}$ is a union of $L_{f(x_2)-\epsilon}$ and $L_{f(x_2)}$ with $x_2\in L_{f(x_2)}$ as in the proof  of Case I,  we can obtain the desired contradiction. If that is not possible, we then consider the connected hypersurface components $L_{f(x_2)}^i$ of $f^{-1}(f(x_2))$ containing $x_2$ with $i=1,\dots, k$. Note that  $\alpha=0$ on $L_{f(x_2)}^i$. Then there exists a connected set $\tilde{\Omega}_{\epsilon}$ with a sufficiently small $\epsilon$ such that $f(x_2)-\epsilon <f(x)<f(x_2)$ for $x\in  \tilde{\Omega}_{\epsilon}$ and either $\partial \tilde{\Omega}_{\epsilon} =L_{f(x_2)-\epsilon}^j\cup L_{f(x_2)}^j$ for some $j$ (see Fig. 2 (a)), or  $\partial \tilde{\Omega}_{\epsilon} =L_{f(x_2)-\epsilon}\cup (\cup_{i=1}^kL_{f(x_2)}^i)$ (see Fig. 2 (b)). Then
$$ 0< \int_{ \tilde{\Omega}_{\epsilon}} (1+f)|z|^2 = -\int_{L_{f(x_2)-\epsilon}^j}\alpha |df| \quad \mbox{or}\quad  -\int_{L_{f(x_2)-\epsilon}}\alpha |df|,$$
implying that $\alpha <0$ on $L_{f(x_2)-\epsilon}^j$ or $L_{f(x_2)-\epsilon}$, which are both impossible by the definition of the point $x_2$.\par
\begin{figure}[htb]
\centering{\scalebox{0.6}{\includegraphics{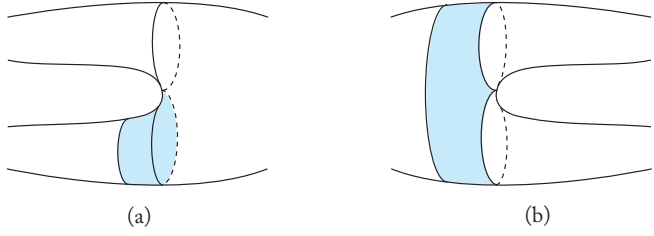}}
\caption{Near a critical point}
\label{figure1-2}}
\end{figure}
The contradictions derived in cases $|df|(x_2)\neq 0$ or $|df|(x_2)=0$ complete the proof of our Lemma.
\end{proof}

Note that ${\mathcal W}_N$ is also continuously well defined on all of $M$ as $\alpha$; if $z(x_0)=0$ for $x_0\in \mbox{\rm Crit}(f)$, ${\mathcal W}_N$ can be defined as zero since $|{\mathcal W}_N|\leq \frac {n-1}{n-2}|z|$, and  if $z(x_1)\neq 0$, $x_1$ is not a critical point of $f$ unless it is a local maximum or minimum point. The differentiation of ${\mathcal W}_N$ on ${\rm {Crit}}(f)$ can also be considered  in the distribution sense.
We can compute the divergence of ${\mathcal W}_N$ as follows.
\begin{lemma}\label{lem2011-4-5-4}
Let $\{E_i\}_{1\leq i \leq n}$ be an orthonormal frame field with $E_n=N$. Then we have
$$
\delta {\mathcal W}_N(E_i) = 0 \quad \mbox{and} \quad \delta {\mathcal W}_N(N) = - \frac{n-2}{n-1} \, \frac{1+f}{|df|} \, |{\mathcal W}_N|^2.
$$
Thus, on $M\setminus \mbox{\rm Crit}(f)$
\be
\delta {\mathcal W}_N = {\delta {\mathcal W}_N(N)} N = \frac{\delta {\mathcal W}_N(N)}{|df|} df. \label{eqn2011-4-6-1}
\ee
\end{lemma}
\begin{proof}
Since $\delta {\mathcal W}=0$, for $X=E_j$
\begin{eqnarray*}
0&=&\delta {\mathcal W}(N,X,N)=-\sum_i D_{E_i}{\mathcal W}(E_i,N,X,N)\\
&=&-\sum_i E_i({\mathcal W}(E_i,N,X,N))+\sum_i [ {\mathcal W}(D_{E_i}E_i,N,X,N) \\
& &+ {\mathcal W}(E_i, D_{E_i}N, X,N)+{\mathcal W}(E_i,N,D_{E_i}X,N)+ {\mathcal W}(E_i,N,X,D_{E_i}N) ].
\end{eqnarray*}
Thus
\begin{eqnarray*}
\delta {\mathcal W}_N(X)&=& -\sum_i D_{E_i}{\mathcal W}_N(E_i,X)\\
&=& -\sum_i E_i({\mathcal W}_N(E_i,X)) +\sum_i ({\mathcal W}_N(D_{E_i}E_i,X)+{\mathcal W}_N(E_i,D_{E_i}X))\\
&=& -\sum_i ({\mathcal W}(E_i, D_{E_i}N, X,N)+ {\mathcal W}(E_i,N,X,D_{E_i}N)).
\end{eqnarray*}
Thus, by (\ref{eqnn02})
\bea (n-2){\mathcal W}(E_i, D_{E_i}N, X,\nabla f) &=& (n-1)df(E_i)z(D_{E_i}N,X)-df(D_{E_i}N)z(E_i,X)\\
& & +z(\nabla f, E_i)g(D_{E_i}N,X)-z(\nabla f, D_{E_i}N)g(E_i,X).
\eea
Note that
$df(E_i)=0$, $$df(D_{E_i}N)=\langle D_{E_i}N,\nabla f\rangle =\frac 1{|df|}(1+f)z(E_i,\nabla f)=0,$$ and $z(\nabla f,E_i)=0$ by Lemma~\ref{lem4}, and finally
$z(\nabla f, D_{E_i}N)= 0$, from the facts that $|df|$ is constant on each level sets of $f$, and
\be
D_{E_i}N = \sum_{j=1}^{n-1} \langle D_{E_i}N, E_j\rangle E_j
= \frac{1}{|df|}\sum_{j=1}^{n-1} \langle D_{E_i}df, E_j\rangle E_j.\label{eqn2011-4-5-3-1}\ee
Therefore ${\mathcal W}(E_i, D_{E_i}N, X,N)=\frac 1{|df|}{\mathcal W}(E_i,D_{E_i}N, X, \nabla f)=0$. Similarly, we have ${\mathcal W}(X,D_{E_i}N,E_i,N)=0$. Hence we may conclude that $\delta {\mathcal W}_N(X)=0$.

 It still remains to show the second identity is correct.
Since $z(N, E_i) = 0$ and ${\mathcal W}_N(N, \cdot) = 0$, we may assume that $E_1, \dots, E_{n-1}$ diagonalize ${\mathcal W}_N$
at a  point $p$. Then, at $p$
\bea
\delta {\mathcal W}_N(N) &=&
-\sum_{i=1}^n D_{E_i}{\mathcal W}_N(E_i, N) = \sum_{i=1}^{n-1} {\mathcal W}_N(E_i, D_{E_i}N)\\
&=&
\sum_{i=1}^{n-1} \langle D_{E_i}N, E_i\rangle {\mathcal W}_N(E_i, E_i).
\eea
Here we used the fact that ${\mathcal W}_N(D_{E_i}E_i,N)=0$.
So, from the CPE,
\bea
\langle D_{E_i}N, E_i\rangle &=&  \frac{1}{|df|} \langle D_{E_i}df, E_i\rangle \nonumber \\
&=&
 \frac{1}{|df|} \left((1+f)z(E_i, E_i) - \frac{sf}{n(n-1)}\right)\nonumber\\
 &=&
 -\frac{1}{(n-1)|df|}\left(\frac{sf}{n} +  (1+f)\a +(n-2)(1+f){\mathcal W}_N(E_i, E_i)\right).\label{eqn2011-5-8-1}
\eea
Hence
\bea
\delta {\mathcal W}_N(N) &=& \sum_{i=1}^{n-1} \langle D_{E_i}N, E_i\rangle {\mathcal W}_N(E_i, E_i) \\
&=&
-\frac{1}{(n-1)|df|}\left(\frac{sf}{n} +  (1+f)\a\right)\sum_{i=1}^{n-1}  {\mathcal W}_N(E_i, E_i) \\
\quad &&- \frac{n-2}{n-1} \frac{1+f}{|df|}\, \sum_{i=1}^{n-1}  {\mathcal W}_N(E_i, E_i)^2\\
 &=& - \frac{n-2}{n-1} \, \frac{1+f}{|df|} \, |{\mathcal W}_N|^2,
\eea
where we used the fact that tr ${\mathcal W}_N=0$ in the last equation.
\end{proof}

\begin{lemma}\label{lem2011-5-7-2}
The differential form $\delta {\mathcal W}_N $ is a closed $1$-form on $M\setminus \mbox{\rm Crit}(f)$.
\end{lemma}
\begin{proof}
Note that from (\ref{eqn2011-4-5-3-1})
\bea
D_{E_i}N = - \frac{1}{(n-1)|df|} \left(\frac{sf}{n}+(1+f)\a +(n-2)(1+f){\mathcal W}_N(E_i, E_i)\right)E_i\label{eqn2011-4-5-5}
\eea
for $i=1, \dots, n-1$. Since ${\mathcal W}_N(E_i, E_i) = -\frac{n-1}{n-2}\left(\frac{\a}{n-1} + z(E_i, E_i)\right)$ by (\ref{zdecomp1}), we obtain
\be
D_{E_i}N = \frac{1}{|df|}\left((1+f)z(E_i, E_i) - \frac{sf}{n(n-1)}\right)E_i.\label{eqn2011-5-7-1}
\ee
In particular, if $i \ne j$,
$$
\langle D_{E_i}E_j, N \rangle = -\langle E_j, D_{E_i}N\rangle = 0
$$
and thus
$$
\langle [E_i, E_j], N\rangle=\langle D_{E_i}E_j - D_{E_j}E_i, N\rangle  =  0.
$$
Hence for a regular value $c$ of $f$, the level set $f^{-1}(c)$ is an integrable hypersurface in $M$.
Since $\delta {\mathcal W}_N(E_i) = 0$ for $1\le  i \le n-1$ by Lemma~\ref{lem2011-4-5-4}, it follows from the Frobenius theorem that
$$
d\delta {\mathcal W}_N(E_i, E_j) = 0 \quad\mbox{for}\quad 1\le i, j \le n-1.
$$
Next, by Lemma~\ref{lem2011-4-5-4} and the fact that $f, |df|$ and $|{\mathcal W}_N|^2$ are all constant on each level sets of $f$, we obtain
\bea
d\delta {\mathcal W}_N(N, E_i)
&=& N(\delta {\mathcal W}_N(E_i)) - E_i(\delta {\mathcal W}_N(N)) - \delta {\mathcal W}_N([N, E_i])\\
&=& - \delta {\mathcal W}_N([N, E_i])  = \delta {\mathcal W}_N(D_{E_i}N - D_N E_i).
\eea
By equation (\ref{eqn2011-4-5-5}) and Lemma~\ref{lem2011-4-5-4},
$\delta {\mathcal W}_N(D_{E_i}N) = 0$
and
\bea
\delta {\mathcal W}_N(D_N E_i) &=& \langle D_N E_i, N\rangle \delta {\mathcal W}_N(N)\\
&=& - \langle  E_i, D_N N\rangle \delta {\mathcal W}_N(N) =0
\eea
since $D_NN = 0$. Thus, we have $d\delta {\mathcal W}_N(N, E_i) = 0$ and consequently $\delta {\mathcal W}_N$ is a closed $1$-form on $M$.
This proves our lemma.\end{proof}

By Lemma~\ref{lem2011-5-7-2} and Proposition~\ref{prop01}, therefore, $d\delta {\mathcal W}_N =0$ almost everywhere.
Now we are ready to prove Theorem~\ref{thmzero}.

Considering $\delta {\mathcal W}_N$ as a vector, it follows from Lemma~\ref{lem2011-4-5-4} that
\bea
\langle \delta Ddf, \delta {\mathcal W}_N\rangle
&=&
\langle -d\Delta f - r(df, \cdot), \delta {\mathcal W}_N\rangle\\
&=&
\frac{s}{n-1}\langle df, \delta {\mathcal W}_N\rangle - r(df, \delta {\mathcal W}_N)\\
&=&
\frac{s}{n-1}\delta {\mathcal W}_N(N)|df| - |df| \delta {\mathcal W}_N(N) r(N, N)\\
&=&
\frac{s}{n-1}\delta {\mathcal W}_N(N)|df| - |df| \delta {\mathcal W}_N(N) \left(\a+\frac{s}{n}\right) \\
&=&
a \left(\a - \frac{s}{n(n-1)}\right) (1+f) |{\mathcal W}_N|^2,
\eea
where $a = \frac{n-2}{n-1}$. Integrating this over $M^0$,
 we have
$$
\int_{M^0} \langle \delta Ddf, \delta {\mathcal W}_N\rangle
 = a\int_{M^0} \left(\a - \frac{s}{n(n-1)}\right) (1+f) |{\mathcal W}_N|^2.
 $$
On the other hand, by Lemma~\ref{lem2011-4-5-4} and Lemma~\ref{lem2011-5-7-2} and the divergence theorem,
\bea
\int_{M^0} \langle \delta Ddf, \delta {\mathcal W}_N\rangle
&=&
\int_{M^0} \langle  Ddf, d\delta {\mathcal W}_N\rangle -\int_{\partial M^0} Ddf(\delta {\mathcal W}_N, N)\\
&=&
-\int_{\partial M^0} \delta {\mathcal W}_N(N) Ddf(N, N)\\
&=&
a\int_{\partial M^0} \frac{1+f}{|df|} |{\mathcal W}_N|^2  Ddf(N, N)=
0,
\eea
where the second equation follows from the fact that $d\delta {\mathcal W}_N=0$ almost everywhere, and
the last equation follows from the definition of $\partial M^0\subset f^{-1}(-1)$.
Thus
$$
\int_{M^0} \left(\a - \frac{s}{n(n-1)}\right) (1+f) |{\mathcal W}_N|^2 = 0.
$$
Since $\a - \frac{s}{n(n-1)} <0$ by Lemma~\ref{analytic}, we may conclude that ${\mathcal W}_N = 0$ on the set $f>-1$.
Integrating the same integrand over $f < -1$, we have ${\mathcal W}_N = 0$ on the whole space $M$. This completes the proof of Theorem~\ref{thmzero}.


\begin{thebibliography}{99}
\bibitem{Besse1987} A.L. Besse, {\it Einstein Manifolds}, New York: Springer-Verlag 1987
\bibitem{BLR2003} L. Bessi\'eres, J. Lafontiane, and L. Rozoy,  Scalar curvature and black holes, preprint
\bibitem{CYH2011} J. Chang, G. Yun, and S. Hwang, Critical point metrics of Total scalar curvature, preprint
\bibitem{Derdzinski1982} A. Derdzinski, On Compact Riemannian Manifolds with Harmonic Curvature, {\it Math. Ann.} {\bf 259}, 145-152 (1982)
\bibitem{FM1974}
A.E. Fischer, J.E. Marsden, Manifolds of Riemannian Metrics
with Prescribed Scalar Curvature, {\it Bull. Am. Math. Soc.} {\bf 80},
479-484 (1974)
\bibitem{Gray1978} A. Gray, Einstein-like manifolds which are not Einstein, {\it Geometriae dedicata}, {\bf 7}, 259-280 (1978)
\bibitem{HW2000}
S. Hwang, Critical points of the scalar curvature functionals on
the space of metrics of constant scalar curvature, {\it manuscripta math}.
{\bf 103}, 135-142 (2000)
\bibitem{Lafontaine1983}
J. Lafontaine, {Sur la g\'eom\'etrie d'une g\'en\'eralisation
de l'\'equation diff\'erentielle d'Obata},  {\it J. Math. Pures
Appliqu\'ees} {\bf 62}, 63-72 (1983)
\bibitem{Morrey1966} C.B. Morrey, {\it Multiple Integrals in the Calculus of Variations}, S\'eminaire de Th\'eorie Spectrale et G\'eom\'etrie, Springer-Verlag Berlin 1966
\bibitem{Obata1962}
M. Obata, Certain conditions for a Riemannian manifold to be
isometric with a sphere, {\it J. Math. Soc. Japan} {\bf 14}(3),  333-340 (1962)
\end{thebibliography}
\end{document}